\documentclass[12pt,reqno]{amsart}
  \usepackage[all,2cell,arrow]{xy}
  \usepackage{amssymb}
 \numberwithin{equation}{section}
 \usepackage{color}
\definecolor{darkgreen}{rgb}{0,0.45,0} 
\usepackage[pagebackref,colorlinks,citecolor=darkgreen,linkcolor=darkgreen]{hyperref}
\usepackage{enumerate,xspace}

\UseAllTwocells
\CompileMatrices

\def\black{\color{black}}

\newcommand{\cb}{\ensuremath{\mathcal B}\xspace}
\newcommand{\cc}{\ensuremath{\mathcal C}\xspace}
\newcommand{\cd}{\ensuremath{\mathcal D}\xspace}

\newcommand{\ct}{\ensuremath{\mathcal T}\xspace}

\newcommand{\cx}{\ensuremath{\mathcal X}\xspace}

\newcommand{\Ab}{\textsf{Ab}\xspace}

\newcommand{\exreg}{\ensuremath{_{\textnormal{ex/reg}}}}

\DeclareMathOperator{\AbSh}{AbSh}

\DeclareMathOperator{\Rel}{Rel}

\newcommand{\Mod}{\textnormal{-$\mathbf{Mod}$}\xspace}

\newcommand{\pid}{\mathfrak{p}}

\newcommand{\two}{\ensuremath{\mathbf{2}\xspace}}

\newcommand{\x}{\times}

\renewcommand{\phi}{\varphi}

\def\1c#1{\stackrel{#1}{\to}}

  \newtheorem{proposition}{Proposition}[section]
  \newtheorem{lemma}[proposition]{Lemma}

  \newtheorem{theorem}[proposition]{Theorem}

  \theoremstyle{definition}
  \newtheorem{definition}[proposition]{Definition}
  \newtheorem{example}[proposition]{Example}

  \theoremstyle{remark}
  \newtheorem{remark}[proposition]{Remark}

  \newcounter{c}
  
  \newcommand{\etyk}[1]{\vspace{-7.4mm}$$\begin{equation}\Label{#1}
  \addtocounter{c}{1}}
  \renewcommand{\]}{\ifnum \value{c}=1 $$\else \end{equation}\fi}
\begin{document}

 \title{Semi-localizations of semi-abelian categories}

\author{Marino Gran}
\address{Institut de Recherche en Math\'ematique et Physique, Universit\'e catholique de Louvain, Chemin du Cyclotron 2, 1348 Louvain-la-Neuve, Belgium}
\email{marino.gran@uclouvain.be}

\author{Stephen Lack}
\address{Department of Mathematics, Macquarie University NSW 2109, Australia}
\email{steve.lack@mq.edu.au}

\thanks{ Lack acknowledges with gratitude the support of the Australian Research Council Discovery Grant DP130101969 and of the Universit\'e catholique de Louvain, and an Australian Research Council Future Fellowship. He is also grateful to the mathematicians in Louvain-la-Neuve for their warm and generous hospitality.}

\date{\today}

\begin{abstract}
A semi-localization of a category is a full reflective subcategory with the property that the reflector is semi-left-exact. There are many interesting examples of semi-localizations, as for instance any torsion-free subcategory of a semi-abelian category. By specializing a result due to S. Mantovani, we first characterize the categories which are semi-localizations of exact Mal'tsev categories. We then prove a new characterization of protomodular categories in terms of binary relations, allowing us to obtain an abstract characterization of the semi-localizations of exact protomodular categories. This result is very useful to study the (hereditarily)-torsion-free subcategories of semi-abelian categories. Some examples are considered in detail in the categories of groups, crossed modules, commutative rings and topological groups. We finally explain how these results extend the corresponding ones obtained in the abelian context by W. Rump. 

\end{abstract}
  
\maketitle


\section{Introduction}\label{sect:torsion-free}

In many areas of mathematics it is useful to {\em localize} a problem, so that one can temporarily focus on behaviour which is at or near a certain ``location'', such as a point in a space or a prime ideal in a ring. 

For a more detailed example, if $\pid$ is a prime ideal in a commutative ring $R$, and $S$ is the complement $R\setminus\pid$, then we may localize at $\pid$ by forming the ring $S^{-1}R$ of fractions. Then restriction along the canonical map $R\to S^{-1}R$ determines a fully faithful functor $S^{-1}R\Mod\to R\Mod$ with a left adjoint sending an $R$-module $M$ to the localized module $S^{-1}M$; furthermore, this left adjoint is an exact functor, and in particular preserves finite limits. Thus we can identify $S^{-1}R\Mod$ with a full reflective subcategory of the category $R\Mod$ of $R$-modules for which the reflector preserves finite limits.

 A similar situation occurs when $X$ is a topological space and $U$ an open set in $X$:  the category $\AbSh(U)$ of sheaves of abelian groups on $U$ can be identified with a full reflective subcategory of the corresponding category $\AbSh(X)$ for which the reflector preserves finite limits. 

\black

Motivated by these and similar examples, a {\em localization} of an abelian category \cc is defined to be a full reflective subcategory \cb for which the reflector $L\colon\cc\to\cb$ \black preserves finite limits (equivalently, for which the reflector is an exact functor).\black  The resulting category \cb is always abelian; conversely, the Gabriel-Popescu theorem \cite{GP} asserts that the Grothendieck abelian categories of \cite{Gr} are precisely those categories which occur as localizations of $R\Mod$ for a commutative ring $R$.

Localizations of abelian categories have been studied intensively, and understood from many points of view. For example, given a localization $L\colon\cc\to\cb$, we may consider the class $\Sigma$ of morphisms $f$ in \cc for which $Lf$ is invertible,
or we may consider the full subcategory \ct consisting of all $C\in\cc$ for which $LC=0$. A full subcategory \ct of \cc arising in this way for some localization of \cc is called a {\em localizing subcategory}. If \cc is locally presentable, and in particular if \cc is the category of modules over some ring, then such localizing subcategories have an elementary characterization, and are in bijective correspondence with localizations of \cc. 

Then again, given the localization $L$, we may define the full subcategory \cx consisting of those $C\in\cc$ for which the reflection $C\to LC$ is a monomorphism. This \cx forms part of a hereditary torsion theory in \cc, in the sense of the following definition, when equipped with the corresponding localizing subcategory \ct.

\begin{definition}
Let \cc be a category with a zero object as well as kernels and cokernels.
A {\em torsion theory} $(\ct,\cx)$ in \cc consists of full (replete) subcategories \ct and \cx of \cc such that 
\begin{itemize}
\item if $T\in\ct$ and $X \in\cx$, then the only morphism from $T$ to $X$ is the zero morphism $T \rightarrow 0 \rightarrow X$;
\item for every object $C\in\cc$ there is a short exact sequence
\begin{equation}\label{shortexact}
\xymatrix{
0 \ar[r] & TC \ar[r]^{t_C}  & C \ar[r]^{{\eta}_C} &  LC   \ar[r] & 0 }
\end{equation}
with $TC \in\ct$ and $LC \in\cx$.
\end{itemize}
The torsion theory is said to be {\em hereditary} if \ct is closed in \cc under subobjects.
\end{definition}

Once again, if \cc is a locally presentable abelian category then localizations on \cc are in bijection with hereditary torsion theories, via the correspondence described above.

The prototypical example of a (hereditary) torsion theory has \cc the category of abelian groups, with $\ct = \mathsf{Ab}_{t.}$ the full subcategory of torsion abelian groups and $\cx = \mathsf{Ab}_{t.f.}$ the full subcategory of torsion-free abelian groups; in this case the short exact sequence \eqref{shortexact} is the one determined by the subgroup $TC$ of torsion elements of the abelian group $C$.  For any prime $p$ there is a radical in the category of abelian groups involving $p$-torsion: indeed, the $p$-primary component $T_p C$ of the subgroup $TC$ of all torsion elements of an abelian group $C$ is a radical, that naturally induces a hereditary torsion theory for each prime $p$.  
By analogy with the classical example $( \mathsf{Ab}_{t.}, \mathsf{Ab}_{t.f.})$ just recalled, for a general torsion theory $( \ct, \cx )$ in an abelian category $\cc$, the subcategory \ct is still called the \emph{torsion subcategory} and the subcategory $\cx$ the \emph{torsion-free subcategory}.

For non-abelian algebraic structures, the interpretation of the concept of localization as a finite limit preserving reflector seems to be highly inadequate even in the most fundamental cases. For instance, in the non-abelian category $\mathsf{Grp}$ of groups there is no non-trivial localization (as explained in \cite{BCGS}, for instance). 
This naturally leads to more subtle notions in order to capture the concept outside the additive context of abelian categories.

Torsion theories were first studied for categories of modules; the case of general abelian categories goes back to \cite{Di}, while they have recently been studied in non-abelian contexts by various authors \cite{BG-torsion, CDT, Gran-Rossi, Rosicky-Tholen, JT-torsion, EverGran, EverGran2}. In particular, they have been studied in \emph{homological categories} \cite{BB}: these are categories which are finitely complete, pointed, regular, and protomodular  \cite{Bourn}. This latter property reduces, in the presence of a zero-object, to the validity of the Split Short Five Lemma. The categories of groups, rings, Lie algebras, loops, crossed modules, ${\mathbb C}^{*}$-algebras, cocommutative Hopf algebras (over a field of characteristic zero) \cite{Kadjo}, and topological groups are all examples of homological categories. 
A torsion-free subcategory $\cx$ of a homological category $\cc$ is again a homological category: this follows from the fact that limits are computed in $\cx$ as in $\cc$, and the factorization $f=ip$ in $\cc$ of any arrow $f$ in $\cx$ as a regular epimorphism $p$ followed by a monomorphism $i$ is also the required factorization of $f$ in $\cx$, since $\cx$ is closed in $\cc$ under subobjects. 

Most of the non-abelian examples of torsion theories which have recently been discovered involve a base category $\cc$ which is not just homological, but \emph{semi-abelian} \cite{JMT}. Recall that a homological category with binary coproducts is semi-abelian when it is also Barr-exact: any equivalence relation in it occurs as the kernel pair of an arrow. All the examples of homological categories mentioned above, except the category of topological groups, are in fact semi-abelian. For instance, the semi-abelian category $\cc = \mathsf{XMod}$ of crossed modules contains the torsion theory $(\ct, \cx) = (\mathsf{Ab}, \mathsf{NormMono})$ (see Example \ref{groupoids}), where $\mathsf{NormMono}$ is its full subcategory of normal monomorphisms and $\mathsf{Ab}$ its full subcategory of abelian groups (an abelian group $A$ is seen here as a crossed module of the form $A \rightarrow 0$).
Further nice examples in the semi-abelian context are provided by the torsion theory $(\mathsf{Nil}, \mathsf{Red})$ of nilpotent rings and reduced rings in the category $\mathsf{CRng}$ of commutative rings (Example \ref{reduced}), or by the torsion theory $(\mathsf{Grp(Conn)}, \mathsf{Grp(Prof)})$ of connected and profinite groups in the category $\mathsf{Grp(Comp)}$ of compact Hausdorff group (see \cite{EverGran}).

 Unlike the homological case, it is no longer true that a torsion-free subcategory of a semi-abelian category is again semi-abelian, as the classical example of torsion-free abelian groups in the category of abelian groups recalled above shows. It is then natural to look for an \emph{abstract characterization of those homological categories which occur as torsion-free subcategories of a semi-abelian category}. This problem is the origin of the present work.

A better understanding of this problem, leading to a solution, can be obtained by looking at it from a more general perspective, and by ``weakening'' the notion of localization to that of \emph{semi-localization}. A remarkable property of the reflection
 $$\xymatrix@=30pt{
{\cc \, } \ar@<1ex>[r]_-{^{\perp}}^-{L} & {\, \cx \, }
\ar@<1ex>[l]^U  }
 $$
where $U \colon \cx \rightarrow \cc$ is the inclusion functor and $L \colon \cc \rightarrow \cx$ the reflector to a torsion-free subcategory, is that the reflector $L$, although not exact, is \emph{semi-left-exact} in the sense of \cite{CHK}. 
This means that the
 reflector $L \colon \cc \rightarrow \cx$ preserves all pullbacks of the form
$$
\xymatrix{P \ar[r] \ar[d] & UX \ar[d]^{Uf}  \\
Q \ar[r]_-{\eta_Q} & ULQ,
}
$$
where $f \colon X \rightarrow LQ$ is an arrow in $\cx$, and $\eta_Q$ is the $Q$-component of the unit $\eta$ of the adjunction. This kind of adjunction plays a central role in Categorical Galois Theory, where the term admissible adjunction is also used (see \cite{BorJan}, for instance). 
Following \cite{Mantovani-semilocalizations, PedicchioRosicky,Rosicky-Tholen}, we call a full reflective subcategory \cx of a category \cc a {\em semi-localization} if the reflection is semi-left exact. 

This significance for us of these semi-localizations is the result of 
\cite{BG-torsion} (see also \cite{CDT, EverGran, JT-torsion}) that the torsion-free subcategories of a homological category are precisely the semi-localizations for which the (components of the) reflector are regular epimorphisms.

Semi-localizations of exact categories have been studied by S. Mantovani in \cite{Mantovani-semilocalizations}, using the exact completion of a regular category  \cite{CarboniMantovani,CatsAlligators, exreg, Cruciani}. This construction is briefly recalled in Section \ref{Exact-regular}, while Mantovani's result together with an application to the study of semi-localizations of exact Mal'tsev categories are given in Section \ref{Mal'tsev}.
In order to study semi-localizations of exact protomodular and of semi-abelian categories it is important to understand the behaviour of the protomodularity property with respect to the construction of the exact completion of a regular category. The key result for this is given in Section \ref{protomodularity}, where a new characterization of protomodularity in terms of binary relations is given (Theorem \ref{caracterisation}), extending a theorem of Z. Janelidze \cite{Zurab-relations3}. This result easily implies that the exact completion of a regular protomodular category is again protomodular (Proposition \ref{completionprotomodular}), showing a difference between the properties of protomodularity and normality (see the last paragraph of \cite{Gran-Janelidze}). The solution of the problem of characterizing torsion-free subcategories of semi-abelian categories is then given in Theorem \ref{thm:semi-abelian-normal}. 
Some examples are explained in detail, and the relationship between the property characterizing semi-localizations of exact categories and the one saying that regular epimorphisms are effective descent morphisms is clarified (see Remarks \ref{remarkDescent} and \ref{Descent}). Section \ref{hereditary} deals with the hereditarily-torsion-free subcategories. The main results are then specialized to the context of abelian categories, recovering in particular two results due to W. Rump \cite{Rump} concerning torsion-free subcategories of abelian categories. 

\subsection*{Acknowledgement}

The first author thanks Tomas Everaert for useful conversations on the subject of this article.

\section{The exact completion of a regular category}\label{Exact-regular}

In this section we recall the fundamental construction which will be crucial to everything that follows, and some of its important properties.

The 2-category of regular categories, regular (=finite limit and regular epimorphism preserving) functors, and natural transformations has a full sub-2-category consisting of the exact categories. The inclusion 2-functor has a left biadjoint, taking a regular category \cx to its exact completion  {\em as a regular category}. We write $\cx\exreg$ for the resulting category; it is quite different to its exact completion as a category with finite limits.

We recall below a construction \cite{perugia,Cruciani,CatsAlligators} of the category $\cx\exreg$ using the calculus of relations. First recall that, for a category \cx with finite limits, a {\em relation} from an object $X$ to an object $Y$ is a subobject of the product $X\x Y$; since these can equally be seen as subobjects of $Y\x X$, any relation $R$ from $X$ to $Y$ determines an {\em opposite relation} $R^\circ$ from $Y$ to $X$. As usual, we identify a morphism $f\colon X\to Y$ with its graph $\binom{1_X}{f}\colon X\to X\x Y$, seen as a relation from $X$ to $Y$. A relation of this form is often called a {\em map}.  A relation $R$ from an object $X$ to itself is called a {\em binary relation} on $X$. Such a binary relation is said to be {\em reflexive} if $1_X\le R$, and {\em symmetric} if $R\le R^\circ$ (or equivalently $R^\circ\le R$, and so $R=R^\circ$).

If \cx is regular, then there is a bicategory $\Rel(\cx)$ whose objects are the objects of \cx, and whose morphisms are the relations between them. A binary relation $R$ is said to be {\em transitive} if the composite $RR$ satisfies $RR\le R$, and an {\em equivalence relation} if it is reflexive, symmetric, and transitive. 
The following properties follow immediately from the definition of composition of relations (see \cite{CLP}, for instance, for further properties of the composition of relations in a regular category):
\begin{lemma}\label{propertiesofrelations}
Let $p \colon X\rightarrow A$ be any arrow in a regular category $\cc$. Then:
\begin{enumerate}[(a)]
  \item{$p^{\circ}p$ is the kernel pair of $p$; }
 \item $p p^{\circ}=1$ { if and only if} $p$
  is a regular {epimorphism}.
\end{enumerate}
\end{lemma}

Equivalence relations are always idempotent in $\Rel(\cx)$. One can freely split these idempotents to obtain a new bicategory $Q(\Rel(\cx))$, and this is once again a bicategory of relations. The category of maps in $Q(\Rel(\cx))$ is then $\cx\exreg$.

For an alternative construction of $\cx\exreg$, based on embedding \cx in the category of sheaves for the regular topology, see \cite{exreg}.

The unit $I\colon\cx\to\cx\exreg$ for the biadjunction is fully faithful, and its image is closed under subobjects. The universal property means in particular that if \cd is any exact category and $F\colon \cx\to\cd$ is a regular functor, then there is an essentially unique regular functor $G\colon \cx\exreg\to\cd$ with $FI\cong G$. 
In fact the construction of $\cx\exreg$ given above is clearly functorial, with respect to regular functors, and so we get a diagram 
$$\xymatrix{
\cx \ar[r]^{F} \ar[d]_{I} & \cd \ar[d]^{I'} \\
\cx\exreg \ar[r]_{F\exreg} & \cd\exreg, }$$
of regular functors which commutes up to isomorphism, in which $I$ and $I'$ are components of the unit of the biadjunction. Since \cd is exact, $I'$ is an equivalence, and composing its inverse with $F\exreg$ we get the required $G$. 

\begin{proposition}
Let \cx be a regular category and \cd an exact category, with $F\colon\cx\to\cd$ a fully faithful finite-limit-preserving functor.
\begin{enumerate}[(a)]
\item If  the image of $F$ is closed under subobjects then $F$ is a regular functor. 
\item In this case, the induced regular functor $G\colon\cx\exreg\to\cd$ is fully faithful.
\end{enumerate}
\end{proposition}

\proof
(a) We must show that $F$ preserves regular epimorphisms. Suppose then that $p\colon X\to Y$ is a regular epimorphism in \cx, and that $Fp\colon FX\to FY$ factorizes through some subobject $D\le FY$; since the image of $F$ is closed under subobjects, this gives a factorization of $p$ through a subobject which is mapped to $D$. Since $p$ is a regular epimorphism, this subobject must be trivial, whence so too is $D$. 

(b) The regular functor $F\colon\cx\to\cd$ induces a homomorphism of bicategories $\Rel(\cx)\to\Rel(\cd)$, which is fully faithful since the image of $F$ is closed under subobjects. Splitting the equivalence relations gives a fully faithful homomorphism $Q(\Rel(\cx))\to Q(\Rel(\cd))$, and now taking maps gives a fully faithful regular functor $\cx\exreg\to\cd\exreg$; but this is just $F\exreg$. 

Since \cd is exact, the unit map $I'\colon\cd\to\cd\exreg$ is an equivalence, and the composite of its inverse with $F\exreg$ is isomorphic to $G$, which is therefore fully faithful.
\endproof

\begin{proposition}\label{prop:exreg-characterization}
 Let \cx be a regular category and \cd an exact category, with $F\colon\cx\to\cd$ a regular functor. The induced regular functor $G\colon \cx\exreg\to\cd$ is an equivalence if and only if $F$ is fully faithful, the image of $F$ is closed under subobjects, and for every object $D\in\cd$ there is a regular epimorphism $q\colon FX\to D$.
\end{proposition}

\proof
If $G$ is an equivalence, then we should prove that $I\colon\cx\to\cx\exreg$ satisfies the conditions in the proposition. 
If $D\in\cx\exreg$, the existence of a regular epimophism $q\colon IX\to D$ follows from the construction given above. If $m\colon D\to IY$ is a subobject of an object $IY$ in the image of $I$, then the composite $mq\colon IX\to IY$ is in the image of $I$. Since $I$ is a regular functor, it must preserve the kernel pair of $mq$, and the quotient of this kernel pair, and this implies that the subobject $D$ is itself in \cx.

Suppose conversely that $F\colon\cx\to\cd$ satisfies the conditions in the proposition. Then $G$ is fully faithful by the previous proposition, and so it will suffice to show that $G$ is essentially surjective on objects. Given $D\in\cd$, let $q\colon FX\to D$ be a regular epimorphism; then the  kernel pair of $q$ is an equivalence relation in $\Rel(\cx)$ on $X$, and so an object of $Q(\Rel(\cx))$, and its image in $Q(\Rel(\cd))$ is isomorphic to $D$. Thus $F\exreg$ is an equivalence of categories, and hence so too is $G$. 
\endproof

\section{Semi-localizations of exact Mal'tsev categories}\label{Mal'tsev}
An abstract characterization of the categories which occur as a semi-localization of an exact category was first discovered by S. Mantovani in \cite{Mantovani-semilocalizations}. 
In this section we begin by giving a (slightly different) proof of Mantovani's result (Theorem \ref{thm:exact}), mainly to make this article more self-contained. We then specialize it to the case of Mal'tsev categories (Theorem \ref{thm:Maltsev}), getting an abstract characterization of semi-localizations of exact Mal'tsev categories. We then observe that topological Mal'tsev algebras form a category which is a semi-localization of an exact Mal'tsev category.

\begin{proposition}\label{semi-localizationimpliesstability}
Let \cx be a semi-localization of an exact category \cc. Then \cx has stable coequalizers of equivalence relations (and so in particular \cx is regular).  
\end{proposition}

\proof
Let us denote by $U \colon\cx\to\cc$ the inclusion functor, and by $L \colon \cc \rightarrow \cx$ its left adjoint. The functor $U$ preserves finite limits, and so equivalence relations in \cx are also equivalence relations in \cc. But \cc is exact, so in particular has coequalizers of equivalence relations, while the reflector $L$, like any left adjoint, preserves coequalizers; thus \cx does have coequalizers of equivalence relations. 

To see that they are stable, consider a diagram
$$\xymatrix{
US \ar@<1ex>[r] \ar@<-1ex>[r] \ar[d] & 
UX' \ar[r] \ar[d] & Q' \ar[r] \ar[d] & UY \ar[d]^{Uf} \\
UR \ar@<1ex>[r] \ar@<-1ex>[r] & 
UX \ar[r] & Q \ar[r]_-{\eta_Q} & ULQ }$$
where $R$ is an equivalence relation in \cx on the object $X$ and $Q$ is the coequalizer in \cc, while $\eta_Q \colon Q\to ULQ$ is the unit of the reflection. Thus the composite $UX\to ULQ$ is the coequalizer in \cx of the equivalence relation $R$. Let $f\colon Y\to LQ$ be a morphism in \cx, then form the rest of the diagram using pullbacks in \cc of $Uf$. By stability of coequalizers of equivalence relations in \cc, $Q'$ is the coequalizer in \cc of the equivalence relation $S$. By semi-left exactness the reflector $L$ inverts the map $Q'\to UY$, which means in turn that $X'\to Y$ is the coequalizer of $S$. This proves the stability of coequalizers of equivalence relations in \cx.
\endproof


\begin{proposition}\label{stablecoequalizersimpliessemileftexact}
If \cx has stable coequalizers of equivalence relations, then the inclusion $U\colon \cx\to\cx\exreg$ has a semi-left-exact left adjoint $L \colon \cx\exreg  \rightarrow \cx$, and the units are regular epimorphisms.
\end{proposition}

\proof
If \cx has coequalizers of equivalence relations then certainly $1\colon\cx\to\cx$ satisfies the conditions for a left Kan extension along $U$ to exist. The resulting functor $L\colon\cx\exreg \rightarrow \cx$ is left adjoint to $U$. 
The components of the adjunction will be regular epimorphisms because \cx is closed in $\cx\exreg$ under subobjects. We need to show that stability of the coequalizers of equivalence relations implies semi-left exactness of the reflector $L \colon\cx\exreg \rightarrow \cx$.

We need to show that a pullback square in $\cx\exreg$ as in the right square of the diagram 
$$\xymatrix{
US \ar@<1ex>[r] \ar@<-1ex>[r] \ar[d] & 
UX' \ar[r] \ar[d] & Q' \ar[r] \ar[d] & UY \ar[d]^{Uf} \\
UR \ar@<1ex>[r] \ar@<-1ex>[r] & 
UX \ar[r] & Q \ar[r]_-{\eta_Q} & ULQ }$$
is preserved by $L$. By the construction of $\cx\exreg$, the object $Q$ appears as the coequalizer of an equivalence relation $R$ in \cx on an object $X$. Complete the rest of the diagram by forming pullbacks. 

Since coequalizers of equivalence relations in $\cx\exreg$ are stable under pullbacks, $Q'$ is the coequalizer in $\cx\exreg$ of the equivalence relation $S$ on $X'$. Since $L$ preserves coequalizers, $LQ'$ is the coequalizer of $S$,
while by stability of coequalizers of equivalence relations $Y$ is also this coequalizer. Thus the comparison $LQ'\to Y$, as in 
$$\xymatrix{
S \ar@<1ex>[r] \ar@<-1ex>[r] \ar[d] & 
X' \ar[r] \ar[d] & LQ' \ar[r] \ar[d] & Y \ar[d] \\
R \ar@<1ex>[r] \ar@<-1ex>[r] & 
X \ar[r] & LQ \ar@{=}[r] & LQ }$$
is invertible, and the reflector $L \colon\cx\exreg \rightarrow \cx$ is indeed semi-left exact. \endproof
We are now ready to prove the following
\begin{theorem}\label{thm:exact}\cite{Mantovani-semilocalizations}
  For a category \cx, the following conditions are equivalent:
  \begin{enumerate}[(a)]
    \item \cx has finite limits and stable coequalizers of equivalence relations (and so in particular is regular);
  \item \cx is a semi-localization of an exact category \cc;
\item \cx is a semi-localization of an exact category \cc, and the units of the reflection are regular epimorphisms;
\item \cx is regular, and is a semi-localization of its exact completion $\cx\exreg$ as a regular category. 
  \end{enumerate}
\end{theorem}
\begin{proof}
Since \cx is closed in $\cx\exreg$ under subobjects, if it is reflective then the units of the reflection are regular epimorphisms. Thus (d) implies (c), and clearly (c) implies (b). The implication (b) implies (a) is given by Proposition \ref{semi-localizationimpliesstability}, and (a) implies (d) by Proposition \ref{stablecoequalizersimpliessemileftexact}.
\end{proof}
\begin{remark}\label{remarkDescent}
Any category \cx satisfying the equivalent conditions of Theorem \ref{thm:exact} has the property that regular epimorphisms are effective for descent. Indeed, regular epimorphisms are always descent morphisms in any regular category (see \cite{JST}, for example). To see that regular epimorphisms are effective for descent in \cx, knowing that they are so in $\cx\exreg$, it suffices to check the following: given any pullback in $\cx\exreg$ along a regular epimorphism $p\colon E \rightarrow B$ in $\cx$
\begin{equation}\label{pullback}
\xymatrix{
E \times_B A \ar[r]^-{\pi_2} \ar[d]_-{\pi_1} & 
A \ar[d]^{f}   \\
E \ar[r]_{p}  & 
B  }
\end{equation}
 the object $A$ belongs to \cx whenever $E\times_B A$ is in \cx (see Corollary $2.7$ in \cite{JT-descent}). For this, it suffices to show that the $A$-component $\eta_A$ of the unit of the adjunction is an isomorphism, and this follows easily from the semi-left exactness of the reflector $L \colon \cx\exreg  \rightarrow \cx$, together with the fact that pulling back along a regular epimorphism in $\cx\exreg$ reflects isomorphisms.
 
 A natural question is then the following: is a regular category with the property that regular epimorphisms are effective descent morphisms necessarily a semi-localization of an exact category? The answer to this question is negative, and an explicit counter-example will be given in Remark \ref{Descent}. 
 \end{remark}

\begin{remark}
  We have characterized, for a category \cx, when there exists some exact category \cc of which \cx is a semi-localization. On the other hand, we cannot hope to recover \cc from \cx, as the following example  shows. Let \cc be the category $\Ab^\two$ whose objects are homomorphisms of abelian groups, and whose morphisms are commutative squares in \Ab. The (image of) the fully faithful functor $\Ab\to\Ab^\two$ sending an abelian group $A$ to the identity morphism $A\to A$ has a left adjoint $\Ab^\two\to\Ab$ sending a morphism to its codomain, and this left adjoint is in fact continuous. Thus \Ab is not just a semi-localization, but a localization of $\Ab^\two$. On  the other hand, \Ab is already exact, so $\Ab\exreg$ is just \Ab itself.  The same example is relevant to all the other characterizations of semi-localizations given in the paper. Remark that there is another fully faithful functor $\Ab\to\Ab^\two$, sending an abelian group $A$ to the morphism $0 \to A$, admitting as left adjoint the cokernel functor $\mathsf{coker} \colon \Ab^\two\to \Ab$, which is semi-left-exact (see Section $1$ in \cite{EG-Homology}, for instance). Accordingly, $\Ab$ can be seen as a semi-localization of $\Ab^\two$ in two different ways.
\end{remark}

In the following sections, we prove variants of Theorem \ref{thm:exact} involving semi-localizations of exact categories with extra properties of some type, such as being protomodular, semi-abelian, or abelian. We are now going to give a prototype for these theorems, involving the structure of Mal'tsev category \cite{CLP}.  Recall that a finitely complete category \cc is a Mal'tsev category if any reflexive relation in \cc is an equivalence relation; or, equivalently, if any reflexive relation in \cc is symmetric.

We are now ready to prove the following:

\begin{theorem}\label{thm:Maltsev}
  For a category \cx, the following conditions are equivalent:
  \begin{enumerate}[(a)]
    \item \cx is a regular Mal'tsev category, and has stable coequalizers of equivalence relations;
  \item \cx is a semi-localization of an exact Mal'tsev category \cc;
\item \cx is a semi-localization of an exact Mal'tsev category \cc, and the units of the reflection are regular epimorphisms;
\item \cx is a regular Mal'tsev category, and is a semi-localization of its exact completion $\cx\exreg$ as a regular category.
  \end{enumerate}
\endproof
\end{theorem}

\proof
Given Theorem~\ref{thm:exact}, we need only show that (b) implies that \cx is a Mal'tsev category, and that if \cx is a Mal'tsev category then so is $\cx\exreg$. The first of these is straightforward: any full subcategory of a Mal'tsev category closed under finite limits is a Mal'tsev category, so in particular this is the case for any full reflective subcategory of a Mal'tsev category. We record the second as the following proposition, whose proof therefore completes the proof of this theorem.
\endproof

\begin{proposition}
  If \cx is a regular Mal'tsev category, then $\cx\exreg$ is also a Mal'tsev category.
\end{proposition}

\proof
Let $R$ be a reflexive relation in $\cx\exreg$ on an object $A$. Cover $A$ with a regular epimorphism $p\colon X\to A$, where $X\in\cx$. Let $S$ be the inverse image $p^\circ Rp$ of $R$ along $p$. Then $1\le p^\circ p\le p^\circ Rp=S$ and so $S$ is also reflexive. Since \cx is a Mal'tsev category, $S$ is then symmetric. But now $R=pp^\circ Rpp^\circ = pSp^\circ\le pS^\circ p^\circ=R^\circ$ and so $R$ is symmetric. This proves that $\cx\exreg$ is also a Mal'tsev category.
\endproof
\begin{example}
Let $\mathbb T$ be a Mal'tsev algebraic theory; that is, an algebraic theory having a ternary term $p(x,y,z)$ satifying the identities $p(x,y,y)=x$ and $p(x,x,y)=y$. The category $\mathbb T (\mathsf{Top})$ of topological Mal'tsev algebras is simply the category of models of the  theory $\mathbb T$ in the category $\mathsf{Top}$ of topological spaces. It is well-known that  $\mathbb T (\mathsf{Top})$ is a regular Mal'tsev category, with regular epimorphisms given by open surjective homomorphisms (see \cite{JP}, for instance). This category certainly has stable coequalizers of equivalence relations, since a coequalizer $\pi \colon X \rightarrow X/R$ of an equivalence relation $R$ on $X$ is computed in $\mathbb T (\mathsf{Top})$ in the same way as in the exact Mal'tsev category $\mathbb T (\mathsf{Set})$ (the corresponding variety of universal algebras) by putting the quotient topology on $X/R$. Since in $\mathbb T (\mathsf{Top})$ quotient maps are open, they are pullback stable, and this gives the result. By Theorem \ref{thm:Maltsev} the category $\mathbb T (\mathsf{Top})$ is then a semi-localization of an exact Mal'tsev category. 
Observe that $\mathbb T (\mathsf{Top})$ is not exact, in general: given an internal equivalence relation $\xymatrix{
R \ar@<1ex>[r] \ar@<-1ex>[r]  & A}$ in $\mathbb T (\mathsf{Top})$ on a topological algebra $A$, it will be effective only when the topology on $R$ is the one induced by the product topology on $A \times A$. On the other hand, it is well known that the category $\mathbb T (\mathsf{Comp})$ of compact Hausdorff algebras is exact for any algebraic theory $\mathbb T$ (see \cite{BC}, for instance), so that ${\mathbb T (\mathsf{Comp})}_{\cx\exreg} = \mathbb T (\mathsf{Comp})$ in that case. 
\end{example}

\section{Protomodularity}\label{protomodularity}

In order to deal with semi-localizations of protomodular categories, we need to reformulate the notion of protomodularity in terms of properties of relations. This new characterization is inspired by the results of Zurab Janelidze \cite{Zurab-relations3} in the pointed context. In the following section we shall use the characterization of protomodularity in Theorem \ref{caracterisation} to prove that if \cx is a regular protomodular category, then $\cx\exreg$ is also protomodular.

Let \cc be a category with pullbacks, and 
$$\xymatrix{
R \ar@<1ex>[r]^{d} \ar@<-1ex>[r]_{c} & X }$$
a relation on $X$ in \cc, seen as a subobject of $X\x X$. If $x$ and $y$ are morphisms $Y\to X$ and the induced map $$\xymatrix{
{\binom{x}{y}} \, \colon Y\ar[r] & X\x X}$$ factors through $R$, we write $xRy$, relying on context to make it clear that we are making an assertion about $x$ and $y$, rather than forming a composite in the bicategory $\Rel(\cc)$.  

\begin{definition}
 We say that a binary relation $R$ is {\em left pseudoreflexive} if there is a morphism $e\colon R\to R$ satisfying $de=ce=d$. $R$ is {\em right pseudoreflexive} if the opposite relation $R^\circ$ is left pseudoreflexive, and {\em pseudoreflexive} if both $R$ and $R^\circ$ are left pseudoreflexive. 
\end{definition}

\begin{remark}
Such an $e$ is clearly unique if it exists, and to say that it exists is to say that $dRd$. More generally, if $xRy$, then we have $x=dz$ and $y=cz$ for a unique $z$, and now $dez=dz=x$ and $cez=dz=x$, and so $xRx$. Thus the condition says that $xRy$ implies $xRx$, and so agrees with the definition of left pseudoreflexive relation given in \cite{Zurab-relations3}. In terms of generalized elements, $e$ maps $(x,y)$ to $(x,x)$. This description makes it clear that $e$ is idempotent; alternatively, this follows form the facts that $dee=de$ and $cee=de=ce$ and the fact that $d$ and $c$ are jointly monomorphic. 
\end{remark}

\begin{definition}\label{pseudosymmetry}
  We say that a morphism $f\colon Y\to X$ exhibits $R$ as {\em left pseudosymmetric} if, when we form the pullback 
$$\xymatrix{
Z \ar[r]^q \ar[d]_p & R \ar[d]^d \\ Y \ar[r]_f & X }$$
there is a morphism $g\colon Z\to R$ with $dg=cq$ and $cg=fp$. We say that $R$ is left pseudosymmetric if there is such an $f$. 
\end{definition}

\begin{remark}
  In terms of generalized elements, this says that $(fz)Ry$ implies $yR(fz)$. 
Clearly if $f$ exhibits $R$ as left pseudosymmetric, then so does any composite $fa$, and in particular so does $fp$. But if $R$ is also left pseudoreflexive, then $ceq=deq=dq=fp$, and so $(fp)R(fp)$. Thus if $R$ is left pseudosymmetric and left pseudoreflexive, then we can choose our $f$ so that $fRf$. 
\end{remark}

\begin{remark}
  In the case where \cc is pointed, we always have $0R0$, and so $R$ is left pseudosymmetric if and only if $0Ry$ implies that $yR0$; in other words, we may always take $f$ to be the zero map $0\to X$. Thus in this pointed context, $R$ is left pseudosymmetric if and only if it is left 0-symmetric in the sense of \cite{Zurab-relations3}. We refer the reader also to \cite{AU}, where the notion of $0$-symmetric relation was first introduced for varieties of universal algebras. 
\end{remark}

\begin{remark}
In the case where \cc is regular, the conditions defining pseudoreflexive and pseudosymmetric relations can be expressed in terms of the calculus of relations. Then $R$ is left pseudoreflexive when $1\cap R^\circ R\le R$, and right pseudoreflexive when $1\cap RR^\circ\le R^\circ$; while $R$ is left pseudosymmetric when $Rf\le R^\circ f$, or equivalently $Rff^\circ\le R^\circ$, for some map $f\colon X\to X$. 
\end{remark}

We now recall a property that is known to characterize protomodular categories, and that we shall take as the definition of protomodularity in the present article:
\begin{definition} \cite{Bourn}
A category \cc with pullbacks is \emph{protomodular} if, given any pullback in \cc
$$\xymatrix{D \ar[r]^g \ar[d]_q  & E \ar[d]_p \\
A \ar[r]_f & B \ar@<-1ex>[u]_s 
}$$
along a split epimorphism $p$ with section $s$, the pair $(g, s)$ is jointly strongly epimorphic.
\end{definition}

\begin{theorem}\label{caracterisation}
  For a category \cc with finite limits, the following conditions are equivalent:
  \begin{enumerate}[(a)]
  \item \cc is protomodular;
\item every relation which is left pseudoreflexive and left pseudosymmetric is symmetric;
\item every relation which is left pseudoreflexive and left pseudosymmetric is pseudoreflexive.
  \end{enumerate}
\end{theorem}

\proof
Suppose (a), and consider a left pseudoreflexive relation 
$$\xymatrix{
R \ar@<1ex>[r]^d \ar@<-1ex>[r]_c & X }$$
with corresponding idempotent $e\colon R\to R$. Split $e$ as 
$$\xymatrix{
R \ar[r]^{r} & I \ar[r]^i & R }$$
with $ri=1$. Now $dir=de=d=ce=cir$, and so $di=ci$; but since $d$ and $c$ are jointly monic and $i$ is monic, it follows that $di$ is itself monic. Furthermore, if $dh=ch$ then $deh=dh$ and $ceh=dh=ch$, and so $eh=h$, and $h$ factorizes through $i$; thus $i$ is the equalizer of $d$ and $c$. 

If $R$ is also left pseudosymmetric, we may choose $f\colon Y\to X$ (in Definition \ref{pseudosymmetry})  in such a way that $fRf$. Thus $f=dh=ch$ for some $h\colon Y\to R$, and so $h=ik$ for some $k\colon Y\to I$. 

We now have  pullbacks 
$$\xymatrix{
Z \ar[r]^{q} \ar[d]_{p} & R \ar[d]_{r} \\ 
Y \ar[r]_k \ar@{=}[d] & I \ar@<-1ex>[u]_{i} \ar[d]^{di}  \\
Y \ar[r]_{f} & X 
}$$ 
with $dir=de=d$, and with $r$ a split epimorphism.  Since \cc is protomodular, $q$ and $i$ are jointly strongly epimorphic. 

Consider the pullback
$$\xymatrix{
S \ar[r]^m \ar[d] & R \ar[d]^{\binom{d}{c}} \\
R \ar[r]_{\binom{c}{d}} & X\x X}$$ 
in which $m$ is clearly monic; this represents the intersection $R\cap R^\circ$. 
Now $f$ exhibits $R$ as left pseudosymmetric, and the corresponding map $g\colon Z\to R$ 
induces a factorization of $q$ through $m$. On the other hand $di=ci$, and so $i$ also factorizes through $m$. But now since $q$ and $i$ are jointly strongly  epimorphic, it follows that $m$ is invertible, and so that $R$ is symmetric.  This proves that (a) implies (b). 

 To see that (b) implies (c), observe that for a symmetric relation, left pseudoreflexivity, right pseudoreflexivity, and pseudoreflexivity are all equivalent. 

To prove that (c) implies (a), suppose that we have a pullback
$$\xymatrix{
D \ar[r]^{g} \ar[d]_{q} & E \ar[d]_{p} \\
A \ar[r]_{f} & B \ar@<-1ex>[u]_{s} }$$
where $p$ has a section $s\colon B\to E$. We must show that $g$ and $s$ are jointly strongly epimorphic. 
Let $m\colon C\to E$ be an arbitrary subobject, and consider the relation $R$ defined by the pullbacks
$$\xymatrix{
R \ar[r]^n \ar[d]_q & P \ar[r]^{p_1} \ar[d]_{p_2} & E \ar[d]^{p} \\
C \ar[r]_m \ar[d]_{m} & E \ar[r]_{p} & B \\
E }$$
(this is indeed a relation because $m$ is a monomorphism). We are going to show that if $g$ and $s$ factorize through $m$, then $m$ is invertible.

 To give $x$ and $y$ satisfying  $xRy$ is equivalently to give $u\colon Y\to C$ and $y\colon Y\to E$ such that $pmu=py$ (with $mu=x$). This clearly implies that $pmu=pmu$; thus $R$ is left pseudoreflexive. 

Next we shall show that $g$ exhibits $R$ as left pseudosymmetric. Now $(gz)Ry$ means that $gz=mu$, for some $u$, and $pgz=py$. Then $py=pgz=fqz$, and so by the universal property of the pullback, there is a unique $v$ with $qv=qz$ and $gv=y$. If $g$ factorizes as $g=mh$, then $y=gv=mhv$. Thus $py=pgz$ and $y=mhv$, and therefore $yR(gz)$; this proves that $R$ is indeed left pseudosymmetric. 

By (c), we may deduce that $R$ is pseudoreflexive, and so in particular right pseudoreflexive. Now $psp=p$, and $sp=mtp$ for some $t$, since $s$ was assumed to factorize through $m$. Thus $(mtp)R1$, and so since $R$ is right pseudoreflexive $1R1$; but this implies that the identity $1\colon E\to E$ factorizes through $m$, and so that $m$ is invertible.  
\endproof

\begin{remark}
The previous theorem can be seen as the extension to non-pointed categories of Theorem $12$ in \cite{Zurab-relations3}. \end{remark}

\section{The semi-abelian case}\label{Semi-abelian}

In this section, we prove various analogues of Theorem~\ref{thm:exact} for exact categories satisfying some additional exactness properties, such as being protomodular or semiabelian. Any category of topological semi-abelian algebras turns out to be a semi-localization of a semi-abelian category. 
\begin{theorem}\label{thm:protomodular}
  For a category \cx, the following conditions are equivalent:
  \begin{enumerate}[(a)]
    \item \cx is regular, protomodular, and has  stable coequalizers of equivalence relations;
  \item \cx is a semi-localization of an exact protomodular category \cc;
\item \cx is a semi-localization of an exact protomodular category \cc, and the units of the reflection are regular epimorphisms;
\item \cx is regular, is a semi-localization of its exact completion $\cx\exreg$ as a regular category, and $\cx\exreg$ is protomodular. 
  \end{enumerate}
\end{theorem}

\proof
Just as in the case of Theorem~\ref{thm:Maltsev}, we need only show that (b) implies that \cx is protomodular, and that $\cx\exreg$ is protomodular if \cx is so.
The first of these follows immediately from the fact that any full subcategory of a protomodular category closed under finite limits is itself protomodular; the second from the following proposition.
\endproof

\begin{proposition}\label{completionprotomodular}
  If \cx is a regular protomodular category, then its exact completion $\cx\exreg$ is also protomodular.
\end{proposition}

\proof
Let $R$ be a left pseudoreflexive and left pseudosymmetric relation in $\cx\exreg$ on an object $A$. Cover $A$ with a regular epimorphism $p\colon X\to A$, where $X\in \cx$.
Now let $S=p^\circ Rp$. This is a relation in \cx, and
\begin{align*}
  1\cap S^\circ S &= 1\cap p^\circ R^\circ pp^\circ Rp \\
  &= 1\cap p^\circ R^\circ Rp \\
  &\le p^\circ p \cap p^\circ R^\circ Rp \\
&= p^\circ(1\cap R^\circ R)p \\
&\le p^\circ R p \\
&= S
\end{align*}
so that $S$ is left pseudoreflexive.

Suppose that there is a morphism $f\colon B\to A$ in $\cx\exreg$ for which $Rf\le R^\circ f$; then we claim that there exists such a morphism with $B\in\cx$. For if $B\notin\cx$, then we may cover it by some $g\colon Y\to B$ with $Y\in\cx$, and now restricting the inequality $Rf\le R^\circ f$ along $g$ we obtain an inequality $Rfg\le R^\circ fg$.

 Suppose then that $B\in\cx$ and that $f\colon B\to A$ satisfies $Rf\le R^\circ f$. Form the pullback
$$\xymatrix{
Y \ar[r]^{g} \ar[d]_{q} & X \ar[d]^{p} \\
B \ar[r]_{f} & A }$$
and note that $Y$ is a subobject of the \cx-object $B\x X$, and so itself  lies in \cx. Now
\begin{align*}
  Sg &= p^\circ R pg \\
  &= p^\circ Rfq \\
  &\le p^\circ R^\circ fq \\
  &= p^\circ R^\circ pg \\
  &= S^\circ g
\end{align*}
and so $S$ is also left pseudosymmetric. Since \cx is protomodular, $S$ is in fact symmetric by Theorem~\ref{caracterisation} . But now
$$R=pp^\circ Rpp^\circ = pSp^\circ=pS^\circ p^\circ = (pSp^\circ)^\circ=R^\circ$$
and so $R$ is symmetric. Thus $\cx\exreg$ is protomodular by Theorem~\ref{caracterisation} once again.
\endproof

We now turn to the case of semi-abelian categories. 
Recall that a homological category is semi-abelian \cite{JMT} when it is also exact, and has binary coproducts. The categories of groups, rings, Lie algebras, crossed modules, and ${\mathbb C}^*$-algebras are all semi-abelian. On the other hand, the categories $\mathsf{Grp}(\mathsf{Top})$ of topological groups and  $\mathsf{Grp}(\mathsf{Haus})$ of Hausdorff groups are homological, but not semi-abelian.
\begin{theorem}\label{thm:semi-abelian}
  For a category \cx, the following conditions are equivalent:
  \begin{enumerate}[(a)]
    \item \cx is homological, and has binary coproducts and  stable coequalizers of equivalence relations;
  \item \cx is a semi-localization of a semi-abelian category \cc;
\item \cx is a semi-localization of a semi-abelian category \cc, and the units of the reflection are regular epimorphisms;
\item \cx is regular, is a semi-localization of its exact completion $\cx\exreg$ as a regular category, and $\cx\exreg$ is semi-abelian.
  \end{enumerate}
\endproof
\end{theorem}

\proof
Once again, the implications $(d)\Rightarrow(c)$ and $(c)\Rightarrow(b)$ are straightforward, and $(b)\Rightarrow(a)$ follows from previous theorems and the fact that any full reflective subcategory of a pointed category with binary coproducts is itself pointed and has binary coproducts. Finally to prove $(a)\Rightarrow(d)$ we should prove that $\cx\exreg$ is pointed and has binary coproducts for any  \cx as in $(a)$.

The zero object of \cx defines a functor $1\to\cx$ which is both left and right adjoint to the unique functor $\cx\to1$. Both of these functors are regular, and so the adjunctions pass to the exact completions. This proves that $\cx\exreg$ has a zero object. 

As for binary coproducts, an argument which appeared in \cite{Gran-SemiabelianExactCompletions} can be adapted to our purposes. Any two objects of $\cx\exreg$ can be presented as coequalizers of equivalence relations
$$\xymatrix @R1pc {
R \ar@<1ex>[r] \ar@<-1ex>[r] & X \ar[r] & A \\
S \ar@<1ex>[r] \ar@<-1ex>[r] & Y \ar[r] & B }$$
with $X$, $Y$, $R$, and $S$ all in \cx. The coproduct
$$\xymatrix{
R+S \ar@<1ex>[r] \ar@<-1ex>[r] & X+Y }$$
need not define an equivalence relation on $X+Y$. On the other hand it does certainly define a reflexive graph, so its regular image is a reflexive relation on $X+Y$. But \cx is protomodular, and  so a Mal'tsev category; thus this reflexive relation is in fact an equivalence relation, and therefore has a coequalizer $X+Y\to C$. Furthermore this map $X+Y\to C$ is clearly also the coequalizer of the reflexive graph defined by $R+S$. Now $C$ is the coproduct of $A$ and $B$ by commutativity of colimits with colimits. 
\endproof

In a pointed protomodular category \cx, a morphism $m\colon X\to Y$ is said to be {\em Bourn-normal} \cite{BournNormal} if there is an equivalence relation $R$ on $Y$ and a discrete fibration
$$\xymatrix{
X\x X \ar[r] \ar@<1ex>[d]^{ \pi_2} \ar@<-1ex>[d]_{ \pi_1} & R \ar@<1ex>[d] \ar@<-1ex>[d] \\
X \ar[r]_m & Y. }$$
It turns out that such an $m$ is necessarily a monomorphism.
For any equivalence relation 
$$\xymatrix{
R \ar@<1ex>[r]^d \ar@<-1ex>[r]_c & Y}$$ 
on $Y$, we may obtain a Bourn-normal monomorphism $m\colon X\to Y$ by composing $c\colon R\to Y$ with the kernel $X\to R$ of $d\colon R\to Y$. This defines an isomorphism between the poset of equivalence relations on $Y$ and the poset of Bourn-normal subobjects of $Y$ \cite[Proposition~3.2.12]{BB}. If the category \cx is exact, then the Bourn-normal morphisms coincide with the normal monomorphisms, but in general this is false. For instance, monomorphisms and Bourn-normal monomorphisms coincide in the category $\mathsf{Ab}(\mathsf{Top})$ of topological abelian groups, but they are not necessarily normal monomorphisms.

 Observe that, in the protomodular context, giving a coequalizer of an equivalence relation is equivalent to giving a cokernel of the corresponding Bourn-normal monomorphism. 

We say that the cokernel $q$ of a morphism $x$  is stable if, for any $f$ as in the diagram below,
$$\xymatrix{
& P \ar[d]^u \ar[r]^{v} & D \ar[d]^f \\
A \ar[r]_x \ar[ur]^{x'} & B \ar[r]_q & C }$$
if we form the pullback $P$ and the unique induced $x'$ for which $ux'=x$ and $vx'=0$, then $v$ is the cokernel of $x'$. 

This allows the following further reformulation:

\begin{proposition}
  A category \cx is a semi-localization of a semi-abelian category if and only if it is homological, has binary coproducts, and every Bourn-normal monomorphism has a stable cokernel. \endproof
\end{proposition}

We may analyze further, in a homological category \cx,  the condition that Bourn-normal monomorphisms have stable cokernels. Suppose then that $m\colon A\to B$ is Bourn-normal. If it has a cokernel $q\colon B\to C$, we may form the kernel $k\colon K\to B$ and the unique factorization $m=kn$ of $m$ through $k$. Now pullback $q$ along $0\to C$ to get a diagram
$$\xymatrix{
& K \ar[r]^0 \ar[d]^k & 0 \ar[d] \\
A \ar[ur]^{n} \ar[r]_{m} & B \ar[r]_q & C. }$$
Stability of the cokernel $q$ implies that $K\to 0$ is the cokernel of $n$. Thus the Bourn-normal monomorphism factorizes as a morphism with trivial cokernel (itself necessarily a Bourn-normal monomorphism) followed by a normal monomorphism.

Suppose conversely that every Bourn-normal $m\colon A\to B$ in a homological category \cx factorizes as a morphism $n\colon A\to K$ with trivial cokernel followed by a normal monomorphism $k\colon K\to B$.
Then $k$ is the kernel of some morphism $f\colon B\to D$, but we can factorize $f$ as a regular (=normal) epimorphism $q\colon B\to C$ followed by a monomorphism $j\colon C\to D$. Now $k$ is the kernel of $f=jq$, but $j$ is a monomorphism and so $k$ is also the kernel of $q$. Since $q$ is a cokernel of some map, it is the cokernel of its kernel $k$. Since $n$ has trivial cokernel, $q$ is also the cokernel of $m=kn$; thus the Bourn-normal $m$ does have a cokernel. 

Now consider a pullback 
$$\xymatrix{
&& B' \ar[d]^{g'} \ar[r]^{q'} & C' \ar[d]^{g} \\
A \ar[r]_n & K \ar[r]_k \ar[ur]^{k'} & B \ar[r]_q & C }$$
of $q$ by an arbitrary morphism $g$, and let  $k'$ be the unique arrow satisfying $q'k'=0$ and $g'k'=k$. Then $k'$ is the kernel of the normal epimorphism $q'$, and so $q'$ is the cokernel of $k'$. 
From the fact that $n$ has trivial cokernel it then follows that $q'$ is also the cokernel of $k'n$. This proves that the cokernel $q$ of $kn=m$ is stable.
 We record this as:

\begin{proposition}
  A category \cx is a semi-localization of a semi-abelian category if and only if it is homological, has binary coproducts, and every Bourn-normal monomorphism factorizes as a morphism with trivial cokernel followed by a normal monomorphism. \endproof
\end{proposition}

Finally, we combine all the characterizations into a single theorem:

\begin{theorem}\label{thm:semi-abelian-normal}
  For a category \cx, the following conditions are equivalent:
  \begin{enumerate}[(a)]
    \item \cx is homological, and has binary coproducts and stable coequalizers of equivalence relations;
  \item \cx is a semi-localization of a semi-abelian category \cc;
\item \cx is a semi-localization of a semi-abelian category \cc, and the units of the reflection are regular epimorphisms;
\item \cx is regular, is a semi-localization of its exact completion $\cx\exreg$ as a regular category, and $\cx\exreg$ is semi-abelian;
\item \cx is homological, has binary coproducts, and stable cokernels of Bourn-normal monomorphisms;
\item \cx is homological, has binary coproducts, and every Bourn-normal monomorphism factorizes as a monomorphism with trivial cokernel followed by a normal monomorphism;
\item \cx is a torsion-free subcategory of a semi-abelian category.
  \end{enumerate}
\end{theorem}

\proof
The equivalence of all conditions but the last has been proved in this section; it remains to recall from Section~\ref{sect:torsion-free} that the torsion-free subcategories of a homological category \cc (and so in particular of a semi-abelian category \cc) are the semi-localizations of \cc whose reflections have regular epimorphic units. 
\endproof
\begin{example}
Let $\mathbb T$ be a semi-abelian algebraic theory: this means that $\mathbb T$ has a unique constant $0$, binary terms $\alpha_i (x,y)$ (for $i \in \{ 1, \cdots , n\}$) and an $(n+1)$-ary term $\beta$ such that the identities $\alpha_i (x,x)=0$ (for $i \in \{ 1, \cdots , n\}$) and 
$\beta( \alpha_1 (x,y), \cdots, \alpha_n (x,y), y)= x$ hold \cite{BJ}.
The category $\mathbb T(\mathsf{Top})$ of topological semi-abelian algebras is homological \cite{BC}, with regular epimorphisms given by open surjective homomorphisms, and this category clearly has binary coproducts (it is actually cocomplete). Given any Bourn-normal monomorphism $n : (N, \tau_N)  \rightarrow (G, \tau_G)$ in $\mathbb T(\mathsf{Top})$, take the canonical quotient $\pi_G \colon (G, \tau_G) \rightarrow (G/N, \tau_q)$, where $G/N$ is the quotient algebra of $G$ by its normal subalgebra $N$, equipped with the quotient topology $\tau_q$. The kernel $k \colon (N, \tau_i) \rightarrow (G, \tau_G)$ of $\pi$ in $\mathsf{\mathbb T}(\mathsf{Top})$ is such that $\tau_i$ is the topology on $N$ induced by $\tau_G$. The canonical factorization $m \colon (N, \tau_N) \rightarrow (N, \tau_i)$ making the following triangle commute is then the identity map in the underlying semi-abelian variety $\mathbb T(\mathsf{Set})$:
$$\xymatrix{
& (N, \tau_i)   \ar[d]^{k} & \\
(N, \tau_N)  \ar[ur]^{m} \ar[r]_{n} & (G, \tau_G) \ar[r]_{\pi} & (G/N, \tau_q)}$$
It follows that any Bourn-normal morphism factorizes as a Bourn-normal monomorphism $m$ with trivial cokernel followed by a normal monomorphism $k$. Any category $\mathbb T(\mathsf{Top})$ of topological semi-abelian algebras is then the semi-localization of a semi-abelian category by Theorem \ref{thm:semi-abelian-normal} (f). A similar argument also applies to the category  $\mathbb T(\mathsf{Haus})$ of Hausdorff semi-abelian algebras, which is then also a semi-localization of a semi-abelian category. Of course, one could take $\mathbb T$ to be the theory of groups, so that in particular this result applies to the categories $\mathsf{Grp} (\mathsf{Top})$ and $\mathsf{Grp} (\mathsf{Haus})$ of topological groups and of Hausdorff groups, respectively.
\end{example}

\begin{example}
Let $\cc=\mathsf{Grp}$ be the category of groups. It is well known that any torsion-free subcategory $\cx$ of $\mathsf{Grp}$ contains all free groups, so that in particular for any group $G$ in $\cx$ there is a normal epimorphism $p \colon X \rightarrow G$,  where $X \in \cx$ (see \cite{EverGran2}, for instance). Since the inclusion of $\cx$ into $\mathsf{Grp}$ is a fully faithful regular functor, and $\cx$ is closed in $\mathsf{Grp}$ under subgroups, it follows  by Proposition~\ref{prop:exreg-characterization} that the exact completion of $\cx$ as a regular category is the category of groups: $\cx\exreg = \mathsf{Grp}$. The same arguments show that the category $\mathsf{Ab}$ of abelian groups is the exact completion of any torsion-free subcategory, as a regular category.   \end{example}

\begin{remark}\label{Descent}
 We now give an example showing that a regular category having the property that regular epimorphisms are effective descent morphisms is not necessarily a semi-localization of an exact category.

Let \cc be the category of abelian groups equipped with a commutative bilinear multiplication, not necessarily associative or unital. This category is a semi-abelian variety. Let \cx be the full subcategory consisting of those objects $A$ with the property that $a^2=0$ implies $a=0$. This is clearly closed in \cc under limits; it is also closed under subobjects and extensions (although not under quotients), so that for any short exact sequence 
 $$
 \xymatrix{0 \ar[r] & K \ar[r] & A \ar[r] & B \ar[r] & 0
 }
 $$
in $\cc$, the object $A$ lies in \cx whenever  both $K$ and $B$ do so. Then \cx is reflective in \cc, each component $\eta_A$ of the unit of the adjunction is a regular epimorphism, and \cx is homological with binary coproducts. It follows that, in pullback  (\ref{pullback}), the object $A$ is in $\cx$ whenever $E \times_B A$, $E$ and $B$ belong to $\cx$, and so that regular epimorphisms in \cx are effective for descent since they are so in the semi-abelian category $\cc$.

Furthermore, the fact that free algebras in \cc lie in \cx implies that any object $C\in\cc$ has a regular epimorphism $X\to C$ with $X\in\cx$, and now by Proposition~\ref{prop:exreg-characterization} the category \cc is the exact completion of the regular category \cx. 

On the other hand, \cx is not a torsion-free subcategory of \cc, as was proved in \cite{JT-torsion}; thus by Theorem \ref{thm:semi-abelian-normal} \cx cannot be a semi-localization of any semi-abelian category. It is not hard to adapt the argument of \cite{JT-torsion} to construct a particular equivalence relation in \cx whose coequalizer is not stable under pullback.

\end{remark}

\section{Hereditary torsion theories}\label{hereditary}

In this section we give a characterization of hereditarily-torsion-free subcategories of semi-abelian categories.

Given a torsion theory $(\ct, \cx)$ in a semi-abelian category $\cc$, the torsion-free subcategory $\cx$ is always closed under subobjects, while the torsion subcategory $\ct$ is closed under quotients. When the torsion subcategory is also closed under subobjects, the torsion theory is said to be {\em hereditary}. In terms of the torsion-free part, the torsion theory is hereditary when the following property holds: if $B\to C$ is a monomorphism and $LC=0$ then also $LB=0$. 
This is due to the fact that the objects in \ct consists of the objects $C$ in \cc for which $LC=0$.  
\black
\begin{theorem}\label{thm:hereditary-semiabelian} For a category \cx the following conditions are equivalent:
\begin{enumerate}[(a)]
\item \cx is a hereditarily-torsion-free subcategory of a semi-abelian category;
\item \cx is a torsion-free subcategory of a semi-abelian category $\cc$ with the property that the reflector $L \colon \cc \rightarrow \cx$ preserves monomorphisms;
\item \cx  is homological, has binary coproducts, every Bourn-normal mono-morphism factorizes as a Bourn-normal monomorphism with trivial cokernel followed by a normal mono\-morphism, and the Bourn-normal monomorphisms with trivial cokernel are stable under pullback.
\end{enumerate}
  \black
\end{theorem}

\proof
The equivalence between the first two conditions is known (see Lemma $5.2$ in \cite{BG-torsion}).

Suppose then that \cx is a torsion-free subcategory of a semi-abelian category \cc, with monomorphism preserving reflector $L \colon \cc \rightarrow \cx$. We are going to prove that Bourn-normal monomorphisms in \cx with trivial cokernel are stable under pullback, so that $(c)$ will hold.

A Bourn-normal monomorphism in \cx is also a Bourn-normal monomorphism in \cc, since the inclusion $\cx\to\cc$ preserves finite limits; but a Bourn-normal monomorphism in the semi-abelian category \cc is in fact a normal monomorphism.

 Let $e\colon A\to B$ be a Bourn-normal monomorphism with trivial cokernel in \cx, and $b\colon B'\to B$ an arbitrary morphism in \cx. Construct the diagram
$$\xymatrix{
A' \ar[r]^{e'} \ar[d]_{a} & B' \ar[r]^{q'} \ar[d]^{b} & 
C' \ar[d]^{c} \\
A \ar[r]_{e} & B \ar[r]_{q} & C }$$
in which the square on the left is a pullback, $q$ is the cokernel in \cc of $e$, and $q'$ and $c$ are obtained via the image factorization of $qb$. Bourn-normal monomorphisms are always stable under pullback; we need to prove that $e'$ has trivial cokernel. 

If $q'f=0$ for a morphism $f$ in \cc, then $qbf=cq'f=0$ and so $bf=eg$ for a unique $g$; thus by the universal property of the pulllback, $f$ factorizes (necessarily uniquely) as $f=e'h$. Thus $e'$ is the kernel of $q'$.  But $q'$ is a regular epimorphism in \cc, thus it is the cokernel of its kernel, namely $e'$. 
The cokernel in \cx of $e$ is (as an object) given by $LC$; we are assuming that this is $0$. The assumption that $L$ preserves monomorphisms implies that $L(c)$ is a monomorphism, and  $LC'=0$.
Since $q'\colon B'\to C'$ is the cokernel in \cc of $e'$, this implies that $e'$ has trivial cokernel (in \cx) as required.

Finally, suppose that \cx satisfies (c), so that  it  is a torsion-free subcategory of some semi-abelian category \cc, and so in particular of $\cx\exreg$, and suppose further that Bourn-normal monomorphisms with trivial cokernel in \cx are stable under pullback. Write $L\colon\cx\exreg\to\cx$ for the reflection. We need to show that if $LA=0$ and $m\colon B\to A$ is a monomorphism, then $LB=0$.

We can cover $A$ via a normal epimorphism $q\colon X\to A$ with $X\in\cx$. Write $k\colon K\to X$ for the kernel of $q$; since \cx is closed under subobjects also $K\in\cx$. Consider the diagram
$$\xymatrix{
& Y \ar[r]^p \ar[d]^{n} & B \ar[d]^{m} \\
K \ar[ur]^{h} \ar[r]_{k} & X \ar[r]_q & A }$$
in which the square is a pullback and $h$ is the unique morphism with $nh=k$ and $ph=0$. By stability of cokernels of Bourn-normal monomorphisms, $p$ is the cokernel of $h$. Since $n$ is a pullback of the monomorphism $m$ it is itself a monomorphism; it follows that $h$ is the pullback of $k$ along $n$. By our assumption, then, $h$ is a Bourn-normal monomorphism with trivial cokernel (in \cx); but this cokernel is $LB$, and so $LB=0$ as required.
\endproof
\begin{example}\label{groupoids}
Let $\cc$ be a semi-abelian category, and write $\mathsf{Eq}(\cc)$ and $\mathsf{Grpd}(\cc)$ for the categories of (internal) equivalence relations and of (internal) groupoids in $\cc$, respectively. The category $\mathsf{Grpd}(\cc)$ is itself semi-abelian \cite{BG-Central}, 
and $\mathsf{Eq}(\cc)$ is easily seen to be a reflective subcategory of $\mathsf{Grpd}(\cc)$ with the property that the units of the reflection are regular epimorphisms. As shown in \cite{BG-torsion}, the category $\mathsf{Eq}(\cc)$ actually is a hereditarily-torsion-free subcategory of $\mathsf{Grpd}(\cc)$, so that it satisfies the equivalent conditions of Theorem \ref{thm:hereditary-semiabelian}. Since for any groupoid $X$ in \cc there is a regular epimorphism $E \rightarrow X$ in $\mathsf{Grpd}(\cc)$ whose domain $E$ is an equivalence relation in $\cc$, it follows that the category $\mathsf{Grpd}(\cc)$ is the exact completion of $\mathsf{Eq}(\cc)$ as a regular category: $\mathsf{Eq}(\cc)\exreg = \mathsf{Grpd}(\cc) $. 

Remark that the category $\mathsf{Grpd}(\cc)$ is a variety (of universal algebras) whenever $\cc$ is a semi-abelian variety, so that $\mathsf{Eq}(\cc)$ is a quasivariety, in this case.
In particular, when $\cc$ is the semi-abelian variety $\mathsf{Grp}$ of groups, the category of equivalence relations becomes the quasivariety $\mathsf{Norm}$ of normal monomorphisms of groups. This category is then a semi-localization of the semi-abelian category $\mathsf{XMod}$ of crossed modules, which is known to be equivalent to the category $\mathsf{Grpd}(\mathsf{Grp})$ of internal groupoids in the category of groups. Accordingly, the category of crossed modules can be seen as the exact completion of the category of normal monomorphisms: $\mathsf{Norm}\exreg = \mathsf{XMod}$.
\end{example}
\begin{example}\label{reduced}
A commutative ring $A$ is \emph{reduced} if the implication $a^n = 0 \Rightarrow a=0$ holds for any element $a \in A$. The category $\mathsf{RedRng}$ of reduced rings is a hereditarily-torsion-free subcategory of the semi-abelian category $\mathsf{CRng}$ of commutative rings (not necessarily with units). Any commutative ring is a quotient of a free commutative ring, which belongs to $\mathsf{RedRng}$, so that 
the full inclusion of $\mathsf{RedRng}$ in $\mathsf{CRng}$ can be seen as the inclusion of $\mathsf{RedRng}$ into its exact completion $\mathsf{RedRng}\exreg$.
The category $\mathsf{RedRng}$ then satisfies the equivalent conditions of Theorem \ref{thm:hereditary-semiabelian}. \end{example}

\section{The abelian case}

We now turn to torsion-free subcategories of abelian categories. 

\begin{theorem}\label{thm:additive}
  For a category \cx, the following conditions are equivalent:
  \begin{enumerate}[(a)]
    \item \cx is regular, additive, and has  stable coequalizers of equivalence relations;
  \item \cx is a semi-localization of an abelian category \cc;
\item \cx is a semi-localization of an abelian category \cc, and the units of the reflection are regular epimorphisms;
\item \cx is regular, is a semi-localization of its exact completion $\cx\exreg$ as a regular category, and $\cx\exreg$ is additive;
\item \cx is regular, additive, and has stable cokernels;
\item \cx is additive, every morphism $f\colon A\to D$ has a factorization 
$$\xymatrix{
A \ar[r]^{q} & B \ar[r]^{g} & C \ar[r]^{k} & D}$$
where $q$ is a regular epimorphism, $g$ is an epimorphism and a monomorphism, and $k$ is a regular monomorphism; and regular (=normal) epimorphisms are pullback stable;
\item \cx is a torsion-free subcategory of an abelian category.
  \end{enumerate}
\endproof
\end{theorem}

\proof
Any full subcategory of an additive category which is closed under finite products will be additive. Thus for the equivalence of the first four conditions it will suffice to prove that if \cx is additive and regular then $\cx\exreg$ is additive. First we show that $\cx\exreg$ is semi-additive. For this it is convenient to use the characterization of semi-additivity given in \cite{non-canonical}: a category with finite products and coproducts is semi-additive if and only if it has a natural family of isomorphisms $X+Y\cong X\x Y$. Equivalently, this says that the diagonal functor $\cx\to\cx\x\cx$ has a left adjoint which is also a right adjoint. But then this adjunction lives in the 2-category of regular categories, and so passing to $\cx\exreg$ we get corresponding adjunctions in the 2-category of exact categories. Thus $\cx\exreg$ is semi-additive. 

In a semi-additive category each hom-set has a natural structure of commutative monoid. The semi-additive category will be additive if and only if each identity morphism has an inverse in the corresponding commutative monoid. Any object $A\in\cx\exreg$ can be covered by an object $X\in\cx$ using a regular epimorphism $p\colon X\to A$. Now $1_X\colon X\to X$ has inverse $-1_X$, since \cx is additive, and the composite $p(-1_X)$ will be inverse to $p$. For any morphism $b \colon B\to X$ we have $pb=0$ if and only if $(-p)b=0$; it follows that $-p$ factorizes uniquely through $p$ as a morphism $-1_A\colon A\to A$, and it is straightforward to show that $-1_A$ is inverse to $1_A$. 

This proves that the first four conditions are all equivalent to the last; we now turn to the remaining two, using Theorem~\ref{thm:semi-abelian-normal}  and the fact that in an additive category the Bourn-normal morphisms are precisely the monomorphisms \cite[Theorem~3.2.16]{BB}. 

Thus in Theorem~\ref{thm:semi-abelian-normal}(e), the condition that Bourn-normal monomorphisms have stable cokernels becomes the condition that all monomorphisms have stable cokernels. But since any morphism factorizes as a regular epimorphism followed by a monomorphism, this is equivalent to the condition that all morphisms have stable cokernels.

Similarly in Theorem~\ref{thm:semi-abelian-normal}(f), we get factorizations of arbitrary mono\-morphisms into a monomorphism with trivial cokernel and a normal (=regular) monomorphism. It remains to observe that in an additive category a morphism has trivial cokernel if and only if it is an epimorphism. 
\endproof

The condition appearing in (f) above says precisely that \cx is {\em almost abelian} in the sense of Rump \cite{Rump}, where the equivalence of (f) and (g) was also proved. As explained in \cite{Rump} the categories of real (or complex) normed vector spaces, Banach spaces (with bounded linear maps as morphisms), and
also the category of locally compact abelian groups are all examples of almost abelian categories.

We can now combine this with Theorem~\ref{thm:hereditary-semiabelian} to give a characterization of hereditarily-torsion-free subcategories of abelian categories.

\begin{theorem}
 A category \cx is a hereditarily-torsion-free subcategory of an abelian category if and only if it is additive and every morphism $f\colon A\to D$ has a factorization 
$$\xymatrix{
A \ar[r]^q & B \ar[r]^g & C \ar[r]^k & D }$$
where $q$ is a regular epimorphism, $g$ is an epimorphism and a monomorphism, and $k$ is a regular monomorphism; and if furthermore, these factorizations are stable under pullback.
\end{theorem}

\proof
Monomorphisms and regular monomorphisms are always stable under pullback; stability of regular epimorphisms is regularity, which was needed  already for a torsion-free subcategory; stability of epimorphisms which are monomorphisms is the extra ingredient from Theorem~\ref{thm:hereditary-semiabelian}.
\endproof

This characterization was also given in \cite{Rump}; there the name {\em integral almost abelian} was used for such a category.
The characterization is simpler than that given in \cite{CarboniMantovani}, but in that paper the authors developed a framework that could be used both for pretoposes and abelian categories. In particular, the assumption made in \cite{CarboniMantovani} that every strong equivalence relation is effective is not needed in the additive context.

\bibliographystyle{plain}

\end{document}